\documentclass[oneside,reqno]{amsart}
\usepackage[utf8]{inputenc}
\usepackage{amsmath}
\usepackage{amsfonts}
\usepackage{amssymb}
\usepackage{setspace}
\usepackage{amssymb}
\usepackage[shortlabels]{enumitem}
\usepackage{xcolor}
\usepackage{graphicx}
\usepackage{amsthm}
\newtheorem{thm}{Theorem}[section]

\numberwithin{equation}{section}
\usepackage{bbm}
\newtheorem{rmk}{Remark}[section]

\newtheorem{prop}{Proposition}[section]
\newtheorem{lm}{Lemma}[section]

\usepackage{geometry}
\usepackage{fancyhdr}
\usepackage{fancyhdr}
\usepackage{hyperref}
\hypersetup{
	colorlinks=true,
	linkcolor=blue,
	filecolor=magenta,      
	urlcolor=blue,
	pdftitle={Overleaf Example},
	citecolor=red,
	pdfpagemode=FullScreen,
}
\usepackage{mfirstuc}
\makeatletter
\def\@setauthors{%
	\begingroup
	\def\thanks{\protect\thanks@warning}%
	\trivlist
	\centering\footnotesize \@topsep30\p@\relax
	\advance\@topsep by -\baselineskip
	\item\relax
	\author@andify\authors
	\def\\{\protect\linebreak}%
	\authors%
	\ifx\@empty\contribs
	\else
	,\penalty-3 \space \@setcontribs
	\@closetoccontribs
	\fi
	\endtrivlist
	\endgroup
}
\date{}
\begin{document}
	\title{Stackelberg-Nash null controllability for stochastic parabolic equations}
	\author{{\large Omar Oukdach\,$^1$, Said Boulite\,$^2$,  Abdellatif Elgrou\,$^3$, and  Lahcen Maniar\,$^{3,4}$}}
	\begin{abstract}
		We study a hierarchical control problem  for  stochastic parabolic equations involving gradient terms. We employ the Stackelberg-Nash strategy with two leaders and two followers. The leaders are responsible for selecting the policy targeting null controllability, while the followers solve  a bi-objective optimal control problem which consists of maintaining the solution process close to prefixed targets. Once the Nash equilibrium is determined, the problem reduces to achieving null controllability of a  coupled forward-backward stochastic system. To solve this problem, via Carleman estimates, we establish a suitable observability inequality. Subsequently, we achieve the desired controllability result.
	\end{abstract}
	\maketitle
	
	\footnotetext[1]{\author{Moulay Ismaïl University of Meknes, FST Errachidia,  MSISI Laboratory, AM2CSI Group, Boutalamine, Errachidia, Morocco}. E-mail: \href{omar.oukdach@gmail.com}{\texttt{omar.oukdach@gmail.com}}}
	\footnotetext[2]{Cadi Ayyad University, National School of Applied Sciences, LMDP, UMMISCO (IRD-UPMC), Marrakesh, Morocco. E-mail: \href{s.boulite@uca.ma}{\texttt{s.boulite@uca.ma}}}
	\footnotetext[3]{Cadi Ayyad University, Faculty of Sciences Semlalia, LMDP, UMMISCO (IRD-UPMC), Marrakesh, Morocco. \\E-mail: \href{abdoelgrou@gmail.com}{\texttt{abdoelgrou@gmail.com}}, \href{maniar@uca.ma}{\texttt{maniar@uca.ma}}}
	\footnotetext[4]{University Mohammed VI  Polytechnic, Vanguard Center, Benguerir, Morocco. E-mail: \href{Lahcen.Maniar@um6p.ma}{\texttt{Lahcen.Maniar@um6p.ma}}}

	\section{Introduction and the main result}
	In this paper, we study a hierarchical bi-objective control problem for linear stochastic heat equations in a bounded domain. Let $T>0$, $G\subset\mathbb{R}^N$ ($N\geq1$) be an open bounded domain with a smooth boundary $\Gamma$ and let the sets $\mathcal{O}$, $\mathcal{O}_1$, $\mathcal{O}_2$, $\mathcal{O}_{1,d}$ and $\mathcal{O}_{2,d}$ be nonempty open subsets of $G$. We denote by $\mathbbm{1}_{\mathcal{E}}$ the characteristic function of a subset $\mathcal{E}\subset G$. Throughout this paper, we use $C$ for a positive constant that may vary from one place to another. The norm and inner product of a given Hilbert space $H$ will be respectively denoted by $|\cdot|_H$ and $\langle\cdot,\cdot\rangle_H$. Set
	$$Q=(0,T)\times G, \,\,\quad \Sigma=(0,T)\times\Gamma\,\,\quad  \text{and}\,\, \quad Q_0=(0,T)\times \mathcal{O}.$$
	
	Let $(\Omega,\mathcal{F},\{\mathcal{F}_t\}_{t\in[0,T]},\mathbb{P})$ be a fixed complete filtered probability space on which a one-dimensional standard Brownian motion $W(\cdot)$ is defined such that $\{\mathcal{F}_t\}_{t\in[0,T]}$ is the natural filtration generated by $W(\cdot)$ and augmented by all the $\mathbb{P}$-null sets in $\mathcal{F}$. For a Banach space $\mathcal{X}$, we denote by $C([0,T];\mathcal{X})$ the Banach space of all $\mathcal{X}$-valued continuous functions defined on $[0,T]$; and $L^2_{\mathcal{F}_t}(\Omega;\mathcal{X})$ denotes the Banach space of all $\mathcal{X}$-valued $\mathcal{F}_t$-measurable random variables $X$ such that $\mathbb{E}\big(\vert X\vert_\mathcal{X}^2\big)<\infty$, with the canonical norm; and $L^2_\mathcal{F}(0,T;\mathcal{X})$ indicates the Banach space consisting of all $\mathcal{X}$-valued $\{\mathcal{F}_t\}_{t\in[0,T]}$-adapted processes $X(\cdot)$ such that $\mathbb{E}\big(\vert X(\cdot)\vert^2_{L^2(0,T;\mathcal{X})}\big)<\infty$, with the canonical norm; and $L^\infty_\mathcal{F}(0,T;\mathcal{X})$ is the Banach space consisting of all $\mathcal{X}$-valued $\{\mathcal{F}_t\}_{t\in[0,T]}$-adapted essentially bounded processes, with its norm denoted simply by $|\cdot|_\infty$; and $L^2_\mathcal{F}(\Omega;C([0,T];\mathcal{X}))$ defines the Banach space consisting of all $\mathcal{X}$-valued $\{\mathcal{F}_t\}_{t\in[0,T]}$-adapted continuous processes $X(\cdot)$ such that $\mathbb{E}\left[\displaystyle\max_{t\in[0,T]}\vert X(t)\vert^2_\mathcal{X}\right]<\infty$, with the canonical norm. In what follows, we adopt the following notations
	$$\mathcal{H}_i=L^2_\mathcal{F}(0,T;L^2(\mathcal{O}_i)),\,\quad \mathcal{H}_{i,d}=L^2_\mathcal{F}(0,T;L^2(\mathcal{O}_{i,d}))\quad\,\textnormal{for}\;\;\;i=1,2,$$
	and
	$$
	\mathcal{U}=L^2_\mathcal{F}(0,T;L^2(\mathcal{O}))\times L^2_\mathcal{F}(0,T;L^2(G)),\,\quad\mathcal{H}=\mathcal{H}_1\times\mathcal{H}_2.
	$$
	We consider the following stochastic parabolic equation 
	\begin{equation}\label{eqq1.1}
		\begin{cases}
			\begin{array}{lll}
				dy - \Delta y \,dt \;\,=&\left[a_1y+B_1\cdot\nabla y+u_1\mathbbm{1}_{\mathcal{O}}+v_1\mathbbm{1}_{\mathcal{O}_1}+v_2\mathbbm{1}_{\mathcal{O}_2}\right] \,dt \\
				&+\left[a_2y+B_2\cdot\nabla y +u_2\right]\,dW(t)&\textnormal{in}\,\,Q,\\
				y=0 &&\textnormal{on}\,\,\Sigma,\\
				y(0)=y_0 &&\textnormal{in}\,\,G,
			\end{array}
		\end{cases}
	\end{equation}
	where $a_1,a_2\in L_\mathcal{F}^\infty(0,T;L^\infty(G))$, $B_1, B_2\in L_\mathcal{F}^\infty(0,T;L^\infty(G;\mathbb{R}^N))$, $y_0\in L^2_{\mathcal{F}_0}(\Omega;L^2(G))$ is the initial state. The leaders and followers are constituted respectively of
	$(u_1,u_2) \in \mathcal{U}$ and $(v_1,v_2) \in \mathcal{H}$. In equation \eqref{eqq1.1}, the set $\mathcal{O}$ represents the primary control domain, while $\mathcal{O}_1$ and $\mathcal{O}_2$ denote secondary control domains, all of which are assumed to be small.  For simplicity, we use only four controls (two leaders and two followers), but very similar considerations hold with more than two followers. The case of a single leader is of great interest; however, this task is far from resolved. See Section \ref{section5} for more explanations.
	
	From \cite{krylov}, it is well-known that \eqref{eqq1.1} is well-posed i.e., for any 
	$y_0\in L^2_{\mathcal{F}_0}(\Omega;L^2(G))$, $(u_1,u_2) \in \mathcal{U}$ and $(v_1,v_2)\in \mathcal{H}$, there exists a unique weak solution 
	$$y\in L^2_\mathcal{F}(\Omega;C([0,T];L^2(G)))\bigcap L^2_\mathcal{F}(0,T;H^1_0(G)).$$
	Moreover, there exists a constant $C>0$ so that
	\begin{align*}\vert y\vert_{L^2_\mathcal{F}(\Omega;C([0,T];L^2(G)))}+\vert y\vert_{L^2_\mathcal{F}(0,T;H^1_0(G))}\leq C\,\left(\vert y_0\vert_{L^2_{\mathcal{F}_0}(\Omega;L^2(G))}+\vert(u_1,u_2)\vert_\mathcal{U}+\vert(v_1,v_2)\vert_\mathcal{H}\right).
	\end{align*}
	
	System \eqref{eqq1.1}, without controls, describes various diffusion phenomena, such as thermal processes. These systems, subject to stochastic disturbances, also incorporate small independent changes during the heat process. For further details, see, for instance, \cite[Chapter 5]{lu2021mathematical} and the references therein.
	\begin{rmk}
		\begin{enumerate}[1.]
			\item Only for the simplicity of the presentation, we have worked with the Laplacian operator in \eqref{eqq1.1}. As in \cite{Preprintelgrou23}, it may be replaced by the following general self-adjoint second-order operator
			\begin{align*}
				L(t)y=\displaystyle\sum_{i,j=1}^N \frac{\partial}{\partial x_i}\left(a_{ij}(t,x)\frac{\partial y}{\partial x_j}\right),
			\end{align*}
			where $a_{ij}:\Omega\times Q\rightarrow\mathbb{R}$ satisfy the following assumptions:
			\begin{enumerate}
				\item $a_{ij}\in L^\infty_\mathcal{F}(\Omega;C^1([0,T];W^{2,\infty}(G)))$ and $a_{ij}=a_{ji}$, for any $1\leq i,j\leq N$.
				\item  There exists a constant $c_0>0$ such that 
				$$\sum_{i,j=1}^N a_{ij}\xi_i\xi_j\geq c_0|\xi|^2\qquad\textnormal{for any}\quad (\omega,t,x,\xi)\in \Omega\times Q\times\mathbb{R}^N.$$
			\end{enumerate}
			It would be also interesting to consider only $W^{1,\infty}$-space regularity of coefficients $a_{ij}$ as in the deterministic case.
			\item As we mentioned above, using only one leader $u_1$ is an open problem. It also seems that considering the diffusion term
			``$\left[a_2y+B_2\cdot\nabla y +\mathbbm{1}_{\mathcal{O}'} u_2\right]\,dW(t)$'' where $\mathcal{O}'$ is another nonempty open subset of $G$, is also unresolved problem; we refer to \cite{Preprintelgrou23,tang2009null} for more explanations.
		\end{enumerate}
	\end{rmk}
	
	This paper deals with Stackelberg-Nash controllability for linear stochastic parabolic equations. To the best of our knowledge, in contrast to the deterministic setting, the present paper is the first to consider hierarchical control problems for stochastic parabolic equations with Dirichlet boundary conditions. Let us now formulate the problem under consideration: For fixed $y_{i,d}\in\mathcal{H}_{i,d}$ ($i=1,2$) two target functions, we consider the following secondary functionals
	\begin{equation*}
		J_i(u_1,u_2;v_1,v_2)=\displaystyle\frac{\alpha_i}{2}\mathbb{E}\iint_{ (0,T)\times \mathcal{O}_{i,d}} |y-y_{i,d}|^2 \,dx\, dt + \displaystyle\frac{\beta_i}{2}\mathbb{E}\iint_{(0,T)\times \mathcal{O}_i} |v_i|^2 \,dx\, dt,\quad i=1,2,
	\end{equation*}
	and the main cost functional
	\begin{equation*}
		J(u_1,u_2)=\displaystyle\frac{1}{2}\mathbb{E}\int_{Q} \left(|\mathbbm{1}_{\mathcal{O}}\,u_1|^2+|u_2|^2\right) \,dx\, dt,
	\end{equation*}
	where $\alpha_i$, $\beta_i$  are positive constants and $y=y(y_0,u_1,u_2,v_1,v_2)$ is the solution of \eqref{eqq1.1}. For a fixed $(u_1,u_2)\in \mathcal{U}$, the pair $(v^{\star}_1,v^{\star}_2)\in\mathcal{H}$ is called a Nash equilibrium  for $(J_1,J_2)$ associated to $(u_1,u_2)$  if 
	\begin{equation*}
		J_1(u_1,u_2;  v^{\star}_1,v^{\star}_2)= \min\limits_{v\in  \mathcal{H}_1}J_1(u_1,u_2; v,v^{\star}_2)\,\,\;\;\textnormal{and}\;\;\,\,\,J_2(u_1,u_2; v^{\star}_1,v^{\star}_2)= \min\limits_{v\in  \mathcal{H}_2}J_2(u_1,u_2;  v^{\star}_1,v). 
	\end{equation*}
	Since the functionals $J_1$ and  $J_2$  are differentiable and  convex, then the pair $(v^{\star}_1,v^{\star}_2)\in\mathcal{H}$ is a Nash equilibrium for $(J_1,J_2)$   if and only if
	\begin{equation}\label{NE7}
		J_1'(u_1,u_2;v^{\star}_1,v^{\star}_2)(v,0)=0\qquad\textnormal{for any}\quad v\in  \mathcal{H}_1, 
	\end{equation}
	and
	\begin{equation}\label{NE72}
		J_2'(u_1,u_2;v^{\star}_1,v^{\star}_2)(0,v)=0\qquad\textnormal{for any}\quad v\in  \mathcal{H}_2.
	\end{equation}
	
	The objective is to prove that for any initial state  $y_0\in L^2_{\mathcal{F}_0}(\Omega;L^2(G))$, there exist a couple of controls  $(u_1,u_2)\in  \mathcal{U}$ minimizing the functional $J$ and an associated Nash equilibrium   $(v^{\star}_1,v^{\star}_2)=(v^{\star}_1(u_1,u_2),v^{\star}_2(u_1,u_2))\in\mathcal{H}$  such that the associated solution of equation \eqref{eqq1.1} satisfies the null controllability
	\begin{equation}\label{ncontrol}
		y(T,\cdot)=0\;\;\textnormal{in}\;\;G,\quad\mathbb{P}\textnormal{-a.s.}
	\end{equation}
	To this end, we follow the Stackelberg-Nash strategy: For each choice of the leaders $(u_1,u_2)$, we look for a Nash equilibrium pair for the functionals $J_i$ ($i=1,2$), meaning finding the controls $v^{\star}_1(u_1,u_2)\in \mathcal{H}_1$ and $v^{\star}_2(u_1,u_2)\in \mathcal{H}_2$, depending on $(u_1,u_2)$ and satisfying \eqref{NE7}-\eqref{NE72}. Once the Nash equilibrium has been identified and fixed for each $(u_1,u_2)$, we proceed to determine controls $(\widehat{u}_1,\widehat{u}_2)\in \mathcal{U}$ such that
	\begin{equation*}\label{eq1.6}
		J(\widehat{u}_1,\widehat{u}_2)= \min_{(u_1,u_2)\in \mathcal{U}} J(u_1,u_2),
	\end{equation*}
	and the associated solution of \eqref{eqq1.1} satisfies the controllability property \eqref{ncontrol}. In the subsequent sections, the entire problem can be reduced to the controllability of a coupled forward-backward stochastic system. This is demonstrated by employing some appropriate Carleman estimates for forward and backward stochastic parabolic equations.
	
	Let us assume the following assumption
	\begin{equation}\label{Assump10}
		\mathcal{O}_d=\mathcal{O}_{1,d}=\mathcal{O}_{2,d} \quad \text{and}\quad \mathcal{O}_d\cap\mathcal{O}\neq\emptyset.
	\end{equation}
	The main result of this paper is stated as follows.
	\begin{thm}\label{th4.1SN}
		Let us assume that the assumption \eqref{Assump10} holds and $\beta_i>0$, $i=1,2$,  are sufficiently large. Then, there exists a positive weight function $\rho =\rho(t)$ blowing up at $t=T$ such that for every target functions $y_{i,d}\in \mathcal{H}_{i,d}$ satisfying
		\begin{equation}\label{inqAss11SN}
			\mathbb{E}\iint_{(0, T)\times \mathcal{O}_{i,d}}\rho^2 |y_{i,d}|^2 \,dx\,dt < \infty,\qquad i=1,2,
		\end{equation}
		and for every initial condition $y_0\in L^2_{\mathcal{F}_0}(\Omega;L^2(G))$, there exist controls $(\widehat{u}_1,\widehat{u}_2)\in \mathcal{U}$  minimizing $J$ and an associated Nash equilibrium $(v^{\star}_1,v^{\star}_2)\in\mathcal{H}$ such that the associated solution $\widehat{y}$ of  \eqref{eqq1.1} satisfies 
		$$\widehat{y}(T,\cdot) =0\;\;\textnormal{in}\;\;G,\quad\mathbb{P}\textnormal{-a.s.}$$
		Moreover, the controls $(\widehat{u}_1,\widehat{u}_2)$ can be chosen so that
		\begin{align*}
			\begin{aligned}
				|(\widehat{u}_1,\widehat{u}_2)|^2_{\mathcal{U}}\leq C\left[\mathbb{E}|y_0|^2_{L^2(G)}+ \sum_{i=1}^{2}\alpha^2_i \mathbb{E}\iint_{(0,T)\times \mathcal{O}_{i,d}} \rho^2 y^2_{i,d} \,dx \,dt\right],
			\end{aligned}
		\end{align*}
		where the positive constant $C$ depending on $G$, $\mathcal{O}_i$, $\mathcal{O}_{i,d}$, $T$, $a_1$, $a_2$, $B_1$ and $B_2$.
	\end{thm}
	\begin{rmk}
		\begin{enumerate}[1.]
			\item It would be quite interesting to consider weaker conditions than \eqref{Assump10}. For instance, studying Stackelberg-Nash controllability for linear stochastic parabolic equations under the geometric condition $\mathcal{O}_{1,d} \cap \mathcal{O} \neq \mathcal{O}_{2,d} \cap \mathcal{O}$ instead of \eqref{Assump10}. We refer to \cite{ArFeGu17} where the authors investigate similar problems for linear deterministic parabolic equations.
			\item The assumption \eqref{inqAss11SN} implies that the target functions $y_{i,d}$, $i=1,2$, approach $0$ as $t \rightarrow T$. This condition ensures that the leaders encounter no obstacles in controlling the system. It remains an open problem to determine whether this assumption is necessary.
		\end{enumerate}
	\end{rmk}
	The problem studied in this paper is considered for deterministic equations in numerous research articles. We do not attempt to give the complete list of existing references. As a substitute, we try to provide some necessary papers. The hierarchical control problem was first introduced by L. Lions for the  heat and wave  equations in  \cite{LiPa} and \cite{LiHy}, respectively. Afterwards, the authors in \cite{D02} and \cite{DL04}, combined the Nash and the Stackelberg strategies in the context of approximate controllability. We also refer to \cite{CF18, GRP02, GMR13, GRP01} for the application of these strategies for different evolution equations. For some other results in the case of exact controllability, we refer for instance to \cite{ArFeGu17, ArCaSa15, AAF, Calsavara, HSP18}. For some results of Stackelberg-Nash controllability for deterministic parabolic equations with dynamic boundary conditions, we refer to \cite{BoMaOuNash, BoMaOuNash2}.
	
	To the authors knowledge, \cite{oukBouElgMan,StNashdeg} are the only papers dealing with Stackelberg-Nash controllability for a class of stochastic parabolic equations: stochastic parabolic equations with dynamic boundary conditions and stochastic degenerate parabolic equations, respectively. For these reasons, the authors consider stochastic parabolic equations an interesting field of investigation for this type of question.\\
	
	The remainder of the paper is structured as follows: Section \ref{section2} deals with the existence, uniqueness, and characterization of the Nash equilibrium. In Section \ref{section3}, we prove the needed Carleman estimates. Section \ref{section4} is devoted to establishing the announced controllability result. Finally, in Section \ref{section5}, we conclude the paper with comments and highlight some open problems.
	
	\section{Nash-equilibrium}\label{section2}
	This section is devoted to demonstrating the existence and uniqueness, as well as characterizing the Nash equilibrium  in the sense
	of \eqref{NE7}-\eqref{NE72} for any $(u_1,u_2)\in \mathcal{U}$.
	\subsection{Existence and uniqueness} The following result shows the existence and uniqueness of the Nah equilibrium for the functionals $(J_1,J_2)$.
	\begin{prop}\label{propp4.1}
		There exists a large $\overline{\beta} >0$ such that, if $\beta_i\geq \overline{\beta}$ for $i = 1, 2$, then for each  $(u_1,u_2)\in \mathcal{U}$, there exists a unique Nash-equilibrium $(v^{\star}_1,v^{\star}_2)=(v^{\star}_1(u_1,u_2),v^{\star}_2(u_1,u_2))\in \mathcal{H}$ for $(J_1, J_2)$ associated to $(u_1,u_2)$. Furthermore, there exists a constant $C>0 $  such that
		\begin{equation}\label{propine4.1}
			|(v^{\star}_1,v^{\star}_2)|_{\mathcal{H}}\leq C\big(1+ |(u_1,u_2)|_{\mathcal{U}}\big),
		\end{equation}
		where the constant $C$ depends on $G$, $\mathcal{O}$, $T$, $\mathcal{O}_i$,  $\mathcal{O}_{i,d}$, $\alpha_i$, $\beta_i$, $y_0$, $a_1$, $a_2$, $B_1$ and $B_2$.
	\end{prop}
	\begin{proof}
		Let us consider the following bounded operators $L_i\in \mathcal{L}(\mathcal{H}_i; L_\mathcal{F}^2(0,T; L^2(G)))$ defined by
		\begin{align*}
			L_i(v_i)=y^i,\qquad i=1,2,
		\end{align*}
		where $y^i$ is the solution of 
		\begin{equation}\label{1.10}
			\begin{cases}
				\begin{array}{ll}
					dy^{i} - \Delta y^i \,dt = [a_1 y^{i}+B_1\cdot\nabla y^i+v_i\mathbbm{1}_{\mathcal{O}_i}] \,dt +[a_2 y^{i}+B_2\cdot\nabla y^i] \,dW(t)&\textnormal{in}\,\,Q,\\
					y^i=0 &\textnormal{on}\,\,\Sigma,\\
					y^{i}(0)=0&\textnormal{in}\,\,G.
				\end{array}
			\end{cases}
		\end{equation}
		Note that the solution $y$ of \eqref{eqq1.1} can be written as follows
		$$y=L_1(v_1)+L_2(v_2)+q,$$
		where $q=q(y_0,u_1,u_2)$ is the solution of
		\begin{equation*}\label{1.132}
			\begin{cases}
				\begin{array}{ll}
					dq - \Delta q \,dt = [a_1 q+B_1\cdot\nabla q+u_1\mathbbm{1}_{\mathcal{O}}] \,dt + [a_2 q+B_2\cdot\nabla q+u_2]\,dW(t)&\textnormal{in}\,\,Q,\\
					q=0 &\textnormal{on}\,\,\Sigma,\\
					q(0)=y_0&\textnormal{in}\,\,G.
				\end{array}
			\end{cases}
		\end{equation*}
		For fixed $(u_1,u_2)\in \mathcal{U}$, we have that for all $v_i\in \mathcal{H}_{i}$
		\begin{align*}
			J_i'(u_1,u_2; v^{\star}_1,v^{\star}_2)(v_i) =&\,\,\alpha_i \big\langle L_1(v^{\star}_1)+L_2(v^{\star}_2) +q-y_{i,d},	L_i(v_i)\big\rangle_{\mathcal{H}_{i,d}}+ \beta_i \big\langle v^{\star}_i,v_i\big\rangle_{\mathcal{H}_{i}},\qquad i=1,2.
		\end{align*}
		Then $(v^{\star}_1,v^{\star}_2)$ is a Nash equilibrium for $(J_1,J_2)$ if and only if	
		\begin{align}\label{4.3nashcara}
			\alpha_i \big\langle L_i^*\big[L_1(v^{\star}_1)+ L_2(v^{\star}_2) -(y_{i,d}-q)\big],	v_i\big\rangle_{\mathcal{H}_{i,d}}+ \beta_i \big\langle v^{\star}_i,v_i\big\rangle_{\mathcal{H}_{i}}=0,\quad\forall v_i\in \mathcal{H}_{i}, \qquad i=1,2,
		\end{align}
		which is equivalent to
		$$\alpha_i   L_i^*\left[  (L_1(v^{\star}_1)+ L_2(v^{\star}_2))\mathbbm{1}_{\mathcal{O}_{i,d}}\right]+ \beta_i v^{\star}_i=\alpha_i  L_i^*((y_{i,d}-q)\mathbbm{1}_{\mathcal{O}_{i,d}})\quad\textnormal{in}\;\;\mathcal{H}_i,\quad i=1,2,$$
		where $L_i^*\in \mathcal{L}(L_\mathcal{F}^2(0,T; L^2(G));\mathcal{H}_i)$ is the adjoint operator of $L_i$. Now, the objective is to show that there exists a unique $(v_1^*,v_2^*)\in\mathcal{H}$ such that
		$$\textbf{M}(v_1^*,v_2^*)=\Phi,$$
		where $\textbf{M}:\mathcal{H}\longrightarrow\mathcal{H}$ is the bounded operator defined by
		$$\textbf{M}(v_1,v_2)=\Big(\alpha_1 L_1^*\left[(L_1(v_1)+L_2(v_2))\mathbbm{1}_{\mathcal{O}_{1,d}}\right]+ \beta_1 v_1\,,\,\alpha_2 L_2^*[(L_1(v_1)+L_2(v_2))\mathbbm{1}_{\mathcal{O}_{2,d}}]+ \beta_2 v_2\Big),$$
		and
		$$\Phi=(\alpha_1  L_1^*((y_{1,d}-q)\mathbbm{1}_{\mathcal{O}_{1,d}}),\alpha_2  L_2^*((y_{2,d}-q)\mathbbm{1}_{\mathcal{O}_{2,d}})).
		$$
		On the other hand, it is easy to see that for any $(v_1,v_2)\in\mathcal{H}$, one has
		\begin{align*}
			\langle\textbf{M}(v_1,v_2),(v_1,v_2)\rangle_\mathcal{H}&=\sum_{i=1}^2\beta_i|v_i|_{\mathcal{H}_i}^2+\sum_{i,j=1}^2\alpha_i\langle L_j(v_j),L_i(v_i)\rangle_{\mathcal{H}_{i,d}}\\
			&\geq \sum_{i=1}^2(\beta_i-C)|v_i|_{\mathcal{H}_i}^2.
		\end{align*}
		Then for a large $\overline{\beta}>0$ so that $\beta_i\geq\overline{\beta}$, $i=1,2$, we conclude that
		\begin{align}\label{2.4corc}
			\langle\textbf{M}(v_1,v_2),(v_1,v_2)\rangle_\mathcal{H}\geq C|(v_1,v_2)|_\mathcal{H}^2.
		\end{align}
		We now introduce the bi-linear functional $\textbf{a}:\mathcal{H}\times\mathcal{H}\rightarrow\mathbb{R}$ as follows
		$$\textbf{a}((v_1,v_2),(\widetilde{v}_1,\widetilde{v}_2))=\langle \textbf{M}(v_1,v_2),(\widetilde{v}_1,\widetilde{v}_2)\rangle_\mathcal{H}\qquad\textnormal{for any}\quad \big((v_1,v_2),(\widetilde{v}_1,\widetilde{v}_2)\big)\in\mathcal{H}\times\mathcal{H},$$
		and also define the following linear continuous functional $\Psi:\mathcal{H}\rightarrow\mathbb{R}$ by
		$$\Psi(v_1,v_2)=\big\langle(v_1,v_2),(\alpha_1 L_1^*((y_{1,d}-q)\mathbbm{1}_{\mathcal{O}_{1,d}}),\alpha_2 L_2^*((y_{2,d}-q)\mathbbm{1}_{\mathcal{O}_{2,d}}))\big\rangle_\mathcal{H}\qquad\textnormal{for any}\quad(v_1,v_2)\in\mathcal{H}.$$
		Notice that $\textbf{a}$ is continuous, and from \eqref{2.4corc}, it is coercive. Therefore, by Lax-Milgram theorem, there exists a unique $(v^*_1,v^*_2)\in\mathcal{H}$ so that
		\begin{align}\label{4.4105}
			\textbf{a}((v^*_1,v^*_2),(v_1,v_2))=\Psi(v_1,v_2)\qquad\textnormal{for any}\quad(v_1,v_2)\in\mathcal{H}.
		\end{align}
		This shows the existence and uniqueness of the Nash equilibrium $(v^*_1,v^*_2)$ for $(J_1,J_2)$ associated to $(u_1,u_2)$. Finally, from \eqref{4.4105}, we have that
		$$|(v^*_1,v^*_2)|_\mathcal{H}\leq C|(\alpha_1 L_1^*((y_{1,d}-q)\mathbbm{1}_{\mathcal{O}_{1,d}}),\alpha_2 L_2^*((y_{2,d}-q)\mathbbm{1}_{\mathcal{O}_{2,d}}))|_\mathcal{H},$$
		which easily implies the inequality \eqref{propine4.1}.
	\end{proof}
	\subsection{Characterization of the Nash-equilibrium}
	Let us now  characterize the Nash equilibrium by the following adjoint systems, $i=1,2$,   
	\begin{equation}\label{backadj}
		\begin{cases} 
			\begin{array}{ll}
				dz^{i}+\Delta z^i \,dt=\big[-a_1 z^{i}-a_2 Z^{i}+\nabla\cdot(z^iB_1+Z^iB_2)-\alpha_i(y-y_{i,d})\mathbbm{1}_{\mathcal{O}_{i,d}}\big] \,dt+Z^{i} \,dW(t) &\textnormal{in}\,\,Q, \\
				z^i=0&\textnormal{on}\,\,\Sigma, \\ z^i(T)=0&\textnormal{in}\,\,G.
			\end{array}
		\end{cases}
	\end{equation}
	By applying Itô formula for solutions of \eqref{eqq1.1} and \eqref{backadj} and integration by parts, one has
	\begin{equation*}
		\alpha_i\big\langle y(u_1,u_2;v_1,v_2)-y_{i,d},     L_i(v_i)\big\rangle_{\mathcal{H}_{i,d}}= \langle z^i, v_i\rangle_{\mathcal{H}_i},\qquad i=1,2,
	\end{equation*}
	which with \eqref{4.3nashcara} yield that $(v^{\star}_1,v^{\star}_2)$ is a Nash-equilibrium if and only if 
	$$\langle z^i, v_i\rangle_{\mathcal{H}_i}+\beta_i\langle v^{\star}_i, v_i\rangle_{\mathcal{H}_i}=0\quad\textnormal{for any}\quad v_i\in \mathcal{H}_i,\qquad i=1,2.$$
	Hence, 
	\begin{equation*}\label{chara.1}
		v^{\star}_i=-\frac{1}{\beta_i} z^i|_{(0,T)\times \mathcal{O}_i}, \qquad i=1,2.
	\end{equation*}
	Therefore, the problem is then reduced to show the  null controllability for solutions of the following coupled forward-backward system, called the optimality system
	\begin{equation}\label{eqq4.7}
		\begin{cases}
			\begin{array}{ll}
				dy - \Delta y \,dt = \left[a_1y+B_1\cdot\nabla y+u_1\mathbbm{1}_{\mathcal{O}}-\displaystyle\sum_{i=1}^2\frac{1}{\beta_i}z^i\mathbbm{1}_{\mathcal{O}_i}\right] \,dt+ \left[a_2y +B_2\cdot\nabla y+u_2\right]\,dW(t)&\textnormal{in}\,\,Q,\\
				dz^{i}+\Delta z^{i} \,dt=\left[-a_1 z^{i}-a_2 Z^{i}+\nabla\cdot(z^iB_1+Z^iB_2)-\alpha_i(y-y_{i,d})\mathbbm{1}_{\mathcal{O}_{i,d}}\right] \,dt+Z^{i} \,dW(t)&\textnormal{in}\,\,Q, \\ 
				y=0, \,\,z^i=0&\textnormal{on}\,\,\Sigma,\\
				y(0)=y_0 &\textnormal{in}\,\, G, \\
				z^i(T)=0, \qquad i=1,2, &\textnormal{in}\,\,G.
			\end{array}
		\end{cases}
	\end{equation}
	By the classical duality argument, the null controllability of \eqref{eqq4.7} is equivalent to show an appropriate observability inequality for the following adjoint backward-forward system
	\begin{equation}\label{ADJSO1}
		\begin{cases}
			d\phi+\Delta \phi \,dt=\left[-a_1 \phi-a_2\Phi+\nabla\cdot(\phi B_1+\Phi B_2)+\displaystyle\sum_{i=1}^2\alpha_i \psi^{i}\mathbbm{1}_{\mathcal{O}_{i,d}}\right] \,dt+\Phi \,dW(t) &\textnormal{in}\,\,Q, \\ 
			d\psi^{i} - \Delta\psi^{i}\,dt = \left[a_1\psi^{i}+B_1\cdot\nabla\psi^i+\frac{1}{\beta_{i}}\mathbbm{1}_{\mathcal{O}_i}(x)\phi\right] \,dt + \left[a_2\psi^{i}+B_2\cdot\nabla\psi^i\right] \,dW(t)&\textnormal{in}\,\,Q,\\
			\phi=0, \,\,\psi^i=0 &\textnormal{on}\,\,\Sigma, \\ 
			\phi(T)=\phi_T&\textnormal{in}\,\,G,\\
			\psi^i(0)=0, \qquad i=1,2, &\textnormal{in}\,\,G.
		\end{cases}
	\end{equation}
	\section{Carleman estimates for the coupled system \eqref{ADJSO1}}\label{section3}
	In this section, we prove some Carleman estimates for the coupled system \eqref{ADJSO1}. Let us first define some weight functions based on the following result, which is proven in \cite{BFurIman}.
	\begin{lm}\label{lmm5.1}
		For any nonempty open subset $\mathcal{O}'\Subset G$, there exists a function $\eta\in C^4(\overline{G})$ such that
		$$
		\eta>0\;\,\, \textnormal{in} \,\,G\,;\qquad \eta=0\;\,\,\, \textnormal{on} \,\,\Gamma;\qquad\vert\nabla\eta\vert>0\; \,\,\,\,\textnormal{in}\,\,\overline{G\setminus \mathcal{O}'}.
		$$
	\end{lm}
	For large parameters $\lambda>1$ and $\mu>1$, we choose the weight functions
	\begin{align*}
		&\,\alpha=\alpha(t,x) = \frac{e^{\mu\eta(x)}-e^{2\mu\vert\eta\vert_\infty}}{t(T-t)},\qquad \varphi=\varphi(t,x) = \frac{e^{\mu\eta(x)}}{t(T-t)},\qquad\theta=e^{\lambda\alpha}.
	\end{align*}
	It is easy to check that there exists a constant $C=C(G,T)>0$ so that
	\begin{align}\label{aligned123}
		\begin{aligned}
			&\,\varphi\geq C,\quad\qquad\vert\varphi_t\vert\leq C\varphi^2,\quad\qquad\vert\varphi_{tt}\vert\leq C\varphi^3,\\
			&\vert\alpha_t\vert\leq Ce^{2\mu\vert\eta\vert_\infty}\varphi^2,\quad\qquad\vert\alpha_{tt}\vert\leq Ce^{2\mu\vert\eta\vert_\infty}\varphi^3.
	\end{aligned}\end{align}
	Let us first introduce the following forward stochastic parabolic equation
	\begin{equation}\label{eqqgfr}
		\begin{cases}
			\begin{array}{ll}
				dz - \Delta z \,dt = (F_1+\nabla\cdot F) \,dt + F_2\,dW(t)&\textnormal{in}\,\,Q,\\
				z=0&\textnormal{on}\,\,\Sigma,\\
				z(0)=z_0&\textnormal{in}\,\,G,
			\end{array}
		\end{cases}
	\end{equation}
	where $z_0\in L^2_{\mathcal{F}_0}(\Omega;L^2(G))$ is the initial state, $F_1, F_2\in L^2_\mathcal{F}(0,T;L^2(G))$ and $F\in L^2_\mathcal{F}(0,T;L^2(G;\mathbb{R}^N))$. We now recall the following Carleman estimate for solutions of equation \eqref{eqqgfr}. The proof is given in \cite[Theorem 3.3]{Preprintelgrou23}.
	\begin{lm}
		There exist a large $\mu_1>1$ such that for all $\mu\geq\mu_1$, one can find constants $C>0$ and $\lambda_1>1$ depending only on $G$, $\mathcal{O}$, $\mu$ and $T$ such that for all $\lambda\geq\lambda_1$, $F_1,F_2\in L^2_\mathcal{F}(0,T;L^2(G))$, $F\in L^2_\mathcal{F}(0,T;L^2(G;\mathbb{R}^N))$ and $z_0\in L^2_{\mathcal{F}_0}(\Omega;L^2(G))$, the solution $z$ of \eqref{eqqgfr} satisfies that
		\begin{align}\label{carfor5.6}
			\begin{aligned}
				&\;\lambda^3\mathbb{E}\int_Q\theta^2\varphi^3 z^2\,dx\,dt+\lambda\mathbb{E}\int_Q \theta^2\varphi|\nabla z|^2\,dx\,dt\\
				&\leq C \Bigg[ \lambda^3\mathbb{E}\int_{Q_0} \theta^2\varphi^3 z^2 \,dx\,dt+ \mathbb{E}\int_Q \theta^2F_1^2 \,dx\,dt\\
				&\hspace{.8cm}+\lambda^2\mathbb{E}\int_Q \theta^2\varphi^2 |F|^2 \,dx\,dt
				+\lambda^2\mathbb{E}\int_Q \theta^2\varphi^2F_2^2 \,dx\,dt\Bigg]. 
		\end{aligned}\end{align}
	\end{lm}
	We need also to consider the following backward stochastic parabolic equation
	\begin{equation}\label{eqqgbc}
		\begin{cases}
			dz+\Delta z\,dt=(F_3+\nabla\cdot F) \,dt+ Z \,dW(t) & \textnormal{in}\,\,Q,\\ 
			z=0 & \textnormal{on}\,\,\Sigma,\\
			z(T)=z_T & \textnormal{in}\,\, G,
		\end{cases}
	\end{equation}
	where $z_T\in L^2_{\mathcal{F}_T}(\Omega;L^2(G))$ is the terminal state, $F_3\in L^2_\mathcal{F}(0,T;L^2(G))$ and $F\in L^2_\mathcal{F}(0,T;L^2(G;\mathbb{R}^N))$. We have the following Carleman estimate for solutions of \eqref{eqqgbc}, see \cite[Theorem 2.2]{Preprintelgrou23} for the proof.
	\begin{lm}
		There exist a large $\mu_2>1$ such that for $\mu\geq\mu_2$, one can find constants $C>0$ and $\lambda_2>1$ depending only on $G$, $\mathcal{O}$, $\mu$ and $T$ such that for all $\lambda\geq\lambda_2$, $F_3\in L^2_\mathcal{F}(0,T;L^2(G))$, $F\in L^2_\mathcal{F}(0,T;L^2(G;\mathbb{R}^N))$ and $z_T\in L^2_{\mathcal{F}_T}(\Omega;L^2(G))$, the solution $(z,Z)$ of \eqref{eqqgbc} satisfies
		\begin{align}\label{carback5.8}
			\begin{aligned}
				&\;\lambda^3\mathbb{E}\int_Q\theta^2\varphi^3 z^2\,dx\,dt+\lambda\mathbb{E}\int_Q \theta^2\varphi|\nabla z|^2\,dx\,dt\\
				&\leq C \Bigg[ \lambda^3\mathbb{E}\int_{Q_0} \theta^2\varphi^3 z^2 \,dx\,dt+ \mathbb{E}\int_Q \theta^2 F_3^2 \,dx\,dt\\
				&\hspace{0.8cm}+\lambda^2\mathbb{E}\int_Q \theta^2\varphi^2 |F|^2 \,dx\,dt+\lambda^2\mathbb{E}\int_Q \theta^2\varphi^2 Z^2 \,dx\,dt\Bigg]. 
			\end{aligned}
		\end{align}
	\end{lm}
	Combining Carleman estimates \eqref{carfor5.6} and \eqref{carback5.8}, we  will first prove the following Carleman estimate for the coupled system \eqref{ADJSO1}.
	\begin{lm}\label{thmm5.1} 
		Let $((\phi,\Phi), \psi^1, \psi^2)$ be the solution of \eqref{ADJSO1} and assume that \eqref{Assump10} holds. Then, there exist a large $\mu_3>1$ such that for $\mu=\mu_3$, one can find constants $C>0$ and $\lambda_3>1$ depending only on $G$, $\mathcal{O}$, $\mu_3$ and $T$ such that for all $\lambda\geq\lambda_3$
		\begin{align}\label{Carlem5.9}
			\begin{aligned}
				&\,\lambda^3\mathbb{E}\int_Q\theta^2\varphi^3 \phi^2\,dx\,dt+\lambda\mathbb{E}\int_Q \theta^2\varphi|\nabla \phi|^2\,dx\,dt\\
				&+\lambda^3\mathbb{E}\int_Q\theta^2\varphi^3 h^2\,dx\,dt+\lambda\mathbb{E}\int_Q \theta^2\varphi|\nabla h|^2\,dx\,dt\\
				&\leq\,C \Bigg[\lambda^7  \mathbb{E}\int_{Q_0} \theta^2 \varphi^7 \phi^2 \, \,dx\,dt+\lambda^5 \mathbb{E}\int_{Q}\theta^2\varphi^5\Phi^2\,dx\,dt\Bigg],
			\end{aligned}
		\end{align}
		where $h= \alpha_1\psi^1+\alpha_2\psi^2$.
	\end{lm}
	\begin{proof}
		Let us choose the set $\mathcal{O}'$ defined in Lemma \ref{lmm5.1} so that 
		\begin{equation*}
			\mathcal{O}'\subset \mathcal{O}'_1\Subset \mathcal{O}\cap \mathcal{O}_d,
		\end{equation*}
		where $\mathcal{O}'_1$ is another arbitrary open subset. We also consider $\zeta\in C^{\infty}(G)$ such that
		\begin{subequations}\label{assmzeta}
			\begin{align}
				&0\leq\zeta\leq 1,\,\quad \zeta =1 \,\, \text{in}\,\, \mathcal{O}',\quad \,\text{supp}(\zeta)\subset \mathcal{O}'_1,\label{assmzeta1}\\ &\quad
				\frac{\Delta\zeta}{\zeta^{1/2}}\in L^\infty(G),\quad
				\frac{\nabla\zeta}{\zeta^{1/2}}\in L^\infty(G;\mathbb{R}^N)\label{assmzeta2}.
			\end{align}
		\end{subequations}
		Such function $\zeta$ exists. In fact, by standard arguments, one can take $\zeta_0\in C^\infty_0(G)$ satisfying \eqref{assmzeta1} and then choose $\zeta=\zeta_0^4$, see, e.g., \cite{HSP18}.
		Firstly, by applying the Carleman estimate \eqref{carfor5.6} for solutions of the system satisfied by $h$, we conclude that there exist a large $\mu_1>1$ such that for $\mu\geq\mu_1$, one can find a constant $C>0$ and a large enough $\lambda_1>1$ so that for all $\lambda\geq\lambda_1$, we obtain
		\begin{align*}
			&\,\lambda^3\mathbb{E}\int_Q\theta^2\varphi^3 h^2\,dx\,dt+\lambda\mathbb{E}\int_Q \theta^2\varphi|\nabla h|^2\,dx\,dt \\ &\leq C\Bigg[\lambda^3 \mathbb{E}\int_0^T\int_{\mathcal{O}'} \theta^2\varphi^3 h^2 \, dx\,dt+\mathbb{E}\int_{Q}\theta^2\Big|\frac{\alpha_1}{\beta_1}\mathbbm{1}_{\mathcal{O}_1}(x)+\frac{\alpha_2}{\beta_2}\mathbbm{1}_{\mathcal{O}_2}(x)\Big|^2\phi^2 \,dx\,dt\Bigg],
		\end{align*}
		which gives that
		\begin{align}\label{Car4.13}
			\begin{aligned}
				&\,\lambda^3\mathbb{E}\int_Q\theta^2\varphi^3 h^2\,dx\,dt+\lambda\mathbb{E}\int_Q \theta^2\varphi|\nabla h|^2\,dx\,dt\\
				&\leq C\Bigg[\lambda^3\mathbb{E}\int_0^T\int_{\mathcal{O}'} \theta^2\varphi^3 h^2 \, dx\,dt
				+ \mathbb{E}\int_{Q} \theta^2\phi^2\, dx\,dt\Bigg].
			\end{aligned}
		\end{align}
		Secondly, using the Carleman estimate \eqref{carback5.8}, we deduce that there exists $\mu_2>1$ such that for $\mu\geq\mu_2$, one can find a constant $C > 0$ and a sufficiently large $\lambda_2>1$ such that for all $\lambda \geq \lambda_2$ 
		\begin{align}\label{car4.14}
			\begin{aligned}
				&\,\lambda^3\mathbb{E}\int_Q\theta^2\varphi^3 \phi^2\,dx\,dt+\lambda\mathbb{E}\int_Q \theta^2\varphi|\nabla \phi|^2\,dx\,dt\\
				&\leq C\Bigg[\lambda^3  \mathbb{E}\int_0^T\int_{\mathcal{O}'} \theta^2\varphi^3 \phi^2 \, dx\,dt+\mathbb{E}\int_{Q} \theta^2 h^2 \, dx\,dt+\lambda^2\mathbb{E}\int_{Q} \theta^2\varphi^2 \Phi^2\,dx\,dt\Bigg].
			\end{aligned}
		\end{align}
		Combining \eqref{Car4.13}-\eqref{car4.14} and taking  $\mu=\mu_3=\max(\mu_1,\mu_2)$ and a large enough $\lambda>1$, we have that
		\begin{align}\label{firsine1}
			\begin{aligned}
				&\,\lambda^3\mathbb{E}\int_Q\theta^2\varphi^3 \phi^2\,dx\,dt+\lambda\mathbb{E}\int_Q \theta^2\varphi|\nabla \phi|^2\,dx\,dt\\
				&+\lambda^3\mathbb{E}\int_Q\theta^2\varphi^3 h^2\,dx\,dt+\lambda\mathbb{E}\int_Q \theta^2\varphi|\nabla h|^2\,dx\,dt\\
				&\leq C\Bigg[ \lambda^3 \mathbb{E}\int_0^T\int_{\mathcal{O}'} \theta^2\varphi^3 \phi^2 \,dx\,dt +\lambda^3 \mathbb{E}\int_0^T\int_{\mathcal{O}'} \theta^2\varphi^3 h^2 \, dx\,dt\\
				&\qquad\;+\lambda^2\mathbb{E}\int_{Q}\theta^2\varphi^2 \Phi^2\,dx\,dt\Bigg].
			\end{aligned}
		\end{align}
		Following this, we shall provide an estimate for the second term on the right-hand side of \eqref{firsine1}. To this end, using Itô's formula, we compute $d(\lambda^3\zeta \theta^2\varphi^3 h\phi)$, integrate the resulting identity over $Q$, and take the mathematical expectation on both sides. This leads us to
		\begin{align*}
			\begin{aligned}
				0&=\mathbb{E}\int_Q d(\lambda^3\zeta \theta^2\varphi^3 h\phi) \,dx\\
				&=\mathbb{E}\int_Q \lambda^3\zeta\theta^2\varphi^3\Bigg\{\phi\Big[\Delta h+a_1h+B_1\cdot\nabla h+\Big(\frac{\alpha_1}{\beta_1}\mathbbm{1}_{\mathcal{O}_1}(x)+\frac{\alpha_2}{\beta_2}\mathbbm{1}_{\mathcal{O}_2}(x)\Big)\phi\Big]+a_2h\Phi\\
				&\quad\quad\;+\Phi B_2\cdot\nabla h+h\big[-\Delta\phi-a_1\phi-a_2\Phi+\nabla\cdot(\phi B_1+\Phi B_2)+h\mathbbm{1}_{\mathcal{O}_d}(x)\big]\Bigg\}\, dx\,dt\\
				&\quad+\mathbb{E}\int_Q (\lambda^3\zeta\theta^2\varphi^3)_th\phi \,dx\,dt\\
				&=\mathbb{E}\int_Q \lambda^3\zeta\theta^2\varphi^3\Bigg\{\phi\Big[\Delta h+B_1\cdot\nabla h+\Big(\frac{\alpha_1}{\beta_1}\mathbbm{1}_{\mathcal{O}_1}(x)+\frac{\alpha_2}{\beta_2}\mathbbm{1}_{\mathcal{O}_2}(x)\Big)\phi\Big]+\Phi B_2\cdot\nabla h\\
				&\quad+h\big[-\Delta\phi+\nabla\cdot(\phi B_1+\Phi B_2)+h\mathbbm{1}_{\mathcal{O}_d}(x)\big]\Bigg\}\, dx\,dt+\mathbb{E}\int_Q (\lambda^3\zeta\theta^2\varphi^3)_th\phi \,dx\,dt.
			\end{aligned}
		\end{align*}
		Therefore, it follows that 
		\begin{align}\label{integ11}
			\begin{aligned}
				\lambda^3\mathbb{E}\int_0^T\int_{\mathcal{O}'} \theta^2\varphi^3 h^2 \,dx\,dt&\leq\lambda^3\mathbb{E}\int_0^T\int_{\mathcal{O}'_1} \zeta \theta^2\varphi^3 h^2 \,dx\,dt \\&=-\lambda^3\mathbb{E}\int_Q \zeta \theta^2\varphi^3\Big(\frac{\alpha_1}{\beta_1}\mathbbm{1}_{\mathcal{O}_1}(x)+\frac{\alpha_2}{\beta_2}\mathbbm{1}_{\mathcal{O}_2}(x)\Big)\phi^2 \,dx\,dt\\
				&\quad\;-\lambda^3\mathbb{E}\int_Q \zeta(\theta^2\varphi^3)_t h\phi \,dx\,dt-\lambda^3\mathbb{E}\int_Q \zeta \theta^2\varphi^3 \phi B_1\cdot\nabla h \,dx\,dt\\
				&\quad\;-\lambda^3\mathbb{E}\int_Q \zeta \theta^2\varphi^3 \Phi B_2\cdot\nabla h \,dx\,dt-\lambda^3\mathbb{E}\int_Q \zeta \theta^2\varphi^3 h\nabla\cdot(\phi B_1+\Phi B_2)\,dx\,dt\\
				&\quad\;-\lambda^3\mathbb{E}\int_Q \zeta \theta^2\varphi^3\phi \Delta h \,dx\,dt+\lambda^3\mathbb{E}\int_Q \zeta \theta^2\varphi^3h\Delta\phi\,dx\,dt.
			\end{aligned}
		\end{align}
		Integrating by parts the fifth term on the right-hand side of \eqref{integ11}, we find that
		\begin{align}\label{integ1101}
			\begin{aligned}
				\lambda^3\mathbb{E}\int_0^T\int_{\mathcal{O}'_1} \zeta \theta^2\varphi^3 h^2 \,dx\,dt &=-\lambda^3\mathbb{E}\int_Q \zeta \theta^2\varphi^3\left[\frac{\alpha_1}{\beta_1}\mathbbm{1}_{\mathcal{O}_1}(x)+\frac{\alpha_2}{\beta_2}\mathbbm{1}_{\mathcal{O}_2}(x)\right]\phi^2 \,dx\,dt\\
				&\quad\;-\lambda^3\mathbb{E}\int_Q \zeta(\theta^2\varphi^3)_t h\phi \,dx\,dt+\lambda^3\mathbb{E}\int_Q \phi h B_1\cdot\nabla(\zeta \theta^2\varphi^3) \,dx\,dt\\
				&\quad\;+\lambda^3\mathbb{E}\int_Q \Phi h B_2\cdot\nabla(\zeta \theta^2\varphi^3) \,dx\,dt-\lambda^3\mathbb{E}\int_Q \zeta \theta^2\varphi^3\phi \Delta h \,dx\,dt\\
				&\quad\;+\lambda^3\mathbb{E}\int_Q \zeta \theta^2\varphi^3h\Delta\phi\,dx\,dt\\
				&:=\sum_{i=1}^6\textbf{I}_i.
			\end{aligned}
		\end{align}
		In the rest of the proof,  fixing $\varepsilon>0$, it is straightforward to see that
		\begin{align}\label{Int2}
			\textbf{I}_1\leq C\lambda^3\mathbb{E}\int_{Q_0}\theta^2\varphi^3\phi^2 \,dx\,dt.
		\end{align}
		Using \eqref{aligned123}, it is easy to see that for a large $\lambda$
		\begin{align}\label{estttheva}
			|(\theta^2\varphi^3)_t|\leq C\lambda\theta^2\varphi^5,
		\end{align}
		which with Young's inequality yield
		\begin{align}\label{Int4}
			\begin{aligned}
				\textbf{I}_2&\leq C\lambda^4\mathbb{E}\int_0^T\int_{\mathcal{O}_1'} \zeta\theta^2\varphi^5|h||\phi| \,dx\,dt\\
				&\leq \varepsilon\lambda^3\mathbb{E}\int_0^T\int_{\mathcal{O}_1'}\zeta \theta^2\varphi^3h^2 \,dx\,dt+C(\varepsilon)\lambda^5\mathbb{E}\int_{Q_0} \theta^2\varphi^7\phi^2 \,dx\,dt.
			\end{aligned}
		\end{align}
		For $I_3$, we have
		\begin{align}\label{414es}
			\textbf{I}_3=\lambda^3\mathbb{E}\int_Q \theta^2\varphi^3\phi h B_1\cdot\nabla\zeta \,dx\,dt+\lambda^3\mathbb{E}\int_Q \zeta\phi h B_1\cdot\nabla(\theta^2\varphi^3) \,dx\,dt. 
		\end{align}
		See that for a large $\lambda$
		\begin{align*}
			|\nabla(\theta^2\varphi^3)|\leq C\lambda\theta^2\varphi^4,
		\end{align*}
		and hence, with  \eqref{assmzeta} and \eqref{414es}, we obtain that
		$$\textbf{I}_3\leq C\lambda^3\mathbb{E}\int_Q \zeta^{1/2}\theta^2\varphi^3|\phi||h| \,dx\,dt+C\lambda^4\mathbb{E}\int_Q \zeta\theta^2\varphi^4 |\phi||h|\,dx\,dt.$$
		Therefore, it follows  that
		\begin{align}\label{estimI3}
			\begin{aligned}
				\textbf{I}_3&\leq \varepsilon\lambda^3\mathbb{E}\int_0^T\int_{\mathcal{O}_1'}\zeta \theta^2\varphi^3h^2 \,dx\,dt+C(\varepsilon)\lambda^3\mathbb{E}\int_{Q_0} \theta^2\varphi^3\phi^2 \,dx\,dt\\
				&\quad+C(\varepsilon)\lambda^5\mathbb{E}\int_{Q_0} \theta^2\varphi^5\phi^2 \,dx\,dt.
			\end{aligned}
		\end{align}
		Similarly for $\textbf{I}_4$, we find  
		\begin{align}\label{estimI4}
			\begin{aligned}
				\textbf{I}_4&\leq \varepsilon\lambda^3\mathbb{E}\int_0^T\int_{\mathcal{O}_1'}\zeta \theta^2\varphi^3h^2 \,dx\,dt+C(\varepsilon)\lambda^3\mathbb{E}\int_{Q} \theta^2\varphi^3\Phi^2 \,dx\,dt\\
				&\quad+C(\varepsilon)\lambda^5\mathbb{E}\int_{Q} \theta^2\varphi^5\Phi^2 \,dx\,dt.
			\end{aligned}
		\end{align}
		By integration by parts and \eqref{assmzeta}, we have that
		\begin{align}\label{ine5.144}
			\begin{aligned}
				\textbf{I}_5+\textbf{I}_6 &= -\lambda^3\mathbb{E}\int_Q \zeta \theta^2\varphi^3\phi \Delta h \,dx\,dt+\lambda^3\mathbb{E}\int_Q \zeta \theta^2\varphi^3h\Delta\phi\,dx\,dt\\
				&=\lambda^3\mathbb{E}\int_Q \Delta(\zeta\theta^2\varphi^3)h\phi\,dx\,dt + 2 \lambda^3\mathbb{E}\int_Q \phi\nabla h\cdot\nabla(\zeta\theta^2\varphi^3) \,dx\,dt\\
				&\leq C\lambda^3\mathbb{E}\int_0^T\int_{\mathcal{O}_1'} \zeta^{1/2}\theta^2\varphi^3|h||\phi|\,dx\,dt+C\lambda^3\mathbb{E}\int_0^T\int_{\mathcal{O}_1'} |\nabla(\theta^2\varphi^3)|\zeta^{1/2}|h||\phi|\,dx\,dt\\
				&\quad\,
				+C\lambda^3\mathbb{E}\int_0^T\int_{\mathcal{O}_1'} |\Delta(\theta^2\varphi^3)|\zeta|h||\phi|\,dx\,dt+C\lambda^3\mathbb{E}\int_0^T\int_{\mathcal{O}_1'} \theta^2\varphi^3|\phi||\nabla h|\,dx\,dt\\
				&\quad\,+C\lambda^3\mathbb{E}\int_0^T\int_{\mathcal{O}_1'} |\nabla(\theta^2\varphi^3)||\phi||\nabla h|\,dx\,dt.
			\end{aligned}
		\end{align}
		Notice that for a large $\lambda$
		\begin{align}\label{5.16inq}
			|\Delta(\theta^2\varphi^3)|\leq C\lambda^2\theta^2\varphi^5.
		\end{align}
		This with estimates above yield
		\begin{align}\label{ine5.16}
			\begin{aligned}
				\textbf{I}_5+\textbf{I}_6&\leq C\lambda^3\mathbb{E}\int_0^T\int_{\mathcal{O}_1'} \zeta^{1/2}\theta^2\varphi^3|h||\phi|\,dx\,dt+C\lambda^4\mathbb{E}\int_0^T\int_{\mathcal{O}_1'} \zeta^{1/2}\theta^2\varphi^4|h||\phi|\,dx\,dt\\
				&\quad+C\lambda^5\mathbb{E}\int_0^T\int_{\mathcal{O}_1'} \zeta\theta^2\varphi^5|h||\phi|\,dx\,dt+C\lambda^3\mathbb{E}\int_0^T\int_{\mathcal{O}_1'} \theta^2\varphi^3|\phi||\nabla h|\,dx\,dt\\
				&\quad+C\lambda^4\mathbb{E}\int_0^T\int_{\mathcal{O}_1'} \theta^2\varphi^4|\phi||\nabla h|\,dx\,dt. 
			\end{aligned}
		\end{align}
		Applying Young's inequality in  the right-hand side of \eqref{ine5.16}, we end up with
		\begin{align}\label{inti13}
			\begin{aligned}
				\textbf{I}_5+\textbf{I}_6&\leq\varepsilon\lambda^3\mathbb{E}\int_0^T\int_{\mathcal{O}_1'} \zeta\theta^2\varphi^3 h^2 dx\,dt+C(\varepsilon)\lambda^3\mathbb{E}\int_{Q_0} \theta^2\varphi^3 \phi^2 dx\,dt\\
				&\quad\,+C(\varepsilon)\lambda^5\mathbb{E}\int_{Q_0} \theta^2\varphi^5 \phi^2 dx\,dt+C(\varepsilon)\lambda^7\mathbb{E}\int_{Q_0} \theta^2\varphi^7 \phi^2 dx\,dt\\
				& \quad\,+\varepsilon\lambda\mathbb{E}\int_Q \theta^2\varphi |\nabla h|^2 dx\,dt.
			\end{aligned}
		\end{align}
		From \eqref{integ11}, \eqref{integ1101}, \eqref{Int2}, \eqref{Int4}, \eqref{estimI3}, \eqref{estimI4} and \eqref{inti13}, we deduce for  large $\lambda$
		\begin{align}\label{Abss11}
			\begin{aligned}
				\lambda^3\mathbb{E}\int_0^T\int_{\mathcal{O}'} \theta^2\varphi^3h^2 \,dx\,dt&\leq 4\varepsilon \lambda^3\mathbb{E}\int_Q \theta^2\varphi^3h^2 \,dx\,dt+C(\varepsilon)\lambda^7\mathbb{E}\int_{Q_0}\theta^2\varphi^7\phi^2 \,dx\,dt\\ 
				&\quad\,+C(\varepsilon)\lambda^5\mathbb{E}\int_{Q}\theta^2\varphi^5\Phi^2 \,dx\,dt+\varepsilon\lambda\mathbb{E}\int_Q \theta^2\varphi |\nabla h|^2 \,dx\,dt.
			\end{aligned}
		\end{align}
		Combining \eqref{firsine1} and \eqref{Abss11} and choosing  a small enough $\varepsilon>0$, we obtain that
		\begin{align}\label{in5.211}
			\begin{aligned}
				&\,\lambda^3\mathbb{E}\int_Q\theta^2\varphi^3 \phi^2\,dx\,dt+\lambda\mathbb{E}\int_Q \theta^2\varphi|\nabla \phi|^2\,dx\,dt\\
				&+\lambda^3\mathbb{E}\int_Q\theta^2\varphi^3 h^2\,dx\,dt+\lambda\mathbb{E}\int_Q \theta^2\varphi|\nabla h|^2\,dx\,dt\\ 
				&\leq C\Bigg[ \lambda^3 \mathbb{E}\int_{Q_0} \theta^2\varphi^3 \phi^2 \,dx\,dt +\lambda^7 \mathbb{E}\int_{Q_0} \theta^2\varphi^7 \phi^2 \, dx\,dt\\
				&\quad\;\;\quad+\lambda^2\mathbb{E}\int_{Q}\theta^2\varphi^2 \Phi^2\,dx\,dt+\lambda^5\mathbb{E}\int_{Q}\theta^2\varphi^5 \Phi^2\,dx\,dt\Bigg].
			\end{aligned}
		\end{align}
		Finally, by choosing a sufficiently large $\lambda$ in \eqref{in5.211}, we deduce the desired Carleman estimate \eqref{Carlem5.9}. This concludes the proof of Lemma \ref{thmm5.1}.
	\end{proof}

	To incorporate the influence of source terms $y_{1,d}$ and $y_{2,d}$ in the optimality system \eqref{eqq4.7}, we show an improved Carleman inequality. For this purpose, we first introduce the following modified weight functions
	
	\begin{align}\label{rec1}
		&\,\overline{\alpha}=\overline{\alpha}(t,x) = \frac{e^{\mu\eta(x)}-e^{2\mu\vert\eta\vert_\infty}}{\ell(t)},\qquad \overline{\varphi}=\overline{\varphi}(t,x) =\frac{e^{\mu\eta(x)}}{\ell(t)},\qquad\overline{\theta}=e^{\lambda\overline{\alpha}},
	\end{align}
	where
	\begin{equation*}\label{eq:adjoint-system} \ell(t)=\begin{cases}
			\begin{array}{ll}
				T^2/4 &\textnormal{for}\;\; 0\leq t\leq T/2,\\
				t(T-t) &\textnormal{for}\;\; T/2\leq t\leq T,
			\end{array}
		\end{cases}
	\end{equation*}
	and the functions
	\begin{equation}\label{rec2}
		\overline{\alpha}^{\star}(t)=\min_{x\in\overline{G}}\overline{\alpha}(t,x)  \quad \text{and} \quad  	\overline{\varphi}^{\star}(t)=\max_{x\in\overline{G}}\overline{\varphi}(t,x).
	\end{equation}
	
	Let us now prove the following main improved  Carleman estimate for the coupled system \eqref{ADJSO1}.
	\begin{thm}\label{lem4.5st} Let $((\phi,\Phi), \psi^1, \psi^2)$ be the solution of \eqref{ADJSO1} and assume that \eqref{Assump10} holds. Then, for $\mu=\mu_3$ given in Lemma \ref{thmm5.1}, there exist constants $C>0$ and $\lambda_4>1$ depending only on $G$, $\mathcal{O}$, $\mu_3$ and $T$ such that for all $\lambda\geq\lambda_4$
		\begin{align}\label{improvedCarl}
			\begin{aligned}
				&\,\mathbb{E} |\phi(0)|^2_{L^2(G)}+ \mathbb{E}\int_Q\overline{\theta}^2\overline{\varphi}^3 \phi^2\,dx\,dt+\mathbb{E}\int_Q\overline{\theta}^2\overline{\varphi} |\nabla\phi|^2\,dx\,dt\\
				&+ \mathbb{E}\int_Q\theta^2\varphi^3 h^2\,dx\,dt+\mathbb{E}\int_Q \theta^2\varphi|\nabla h|^2\,dx\,dt\\
				&\leq C \Bigg[\lambda^7  \mathbb{E}\int_{Q_0} \theta^2 \varphi^7 \phi^2 \, \,dx\,dt+\lambda^5 \mathbb{E}\int_{Q}\theta^2\varphi^5\Phi^2\,dx\,dt\Bigg],
			\end{aligned}
		\end{align}
		where $h= \alpha_1\psi^1+\alpha_2\psi^2$.
	\end{thm}
	\begin{proof}
		Let us consider the function $\kappa\in C^1([0,T])$ such that
		\begin{equation}\label{kappadef}
			\kappa=1\quad \text{in}\,\; [0,T/2],\quad \ \,\,\kappa= 0\quad \text{in} \,\,\;[3T/4, T]\quad \text{and} \,\,\;\;\kappa_t\leq C/T^2.
		\end{equation}
		Set $(p,P)=\kappa(\phi,\Phi)$, then $(p,P)$ is the solution of the system 
		\begin{equation}\label{ADJSO}
			\begin{cases}
				dp+\Delta p \,dt=\left[-a_1 p-a_2 P+\nabla\cdot(pB_1+PB_2)+\kappa h\mathbbm{1}_{\mathcal{O}_d}(x)+ \kappa_t\phi\right] \,dt+P \,dW(t) & \textnormal{in } Q, \\ 
				p=0 & \textnormal{on } \Sigma, \\ 
				p(T)=0 & \textnormal{in } G .
			\end{cases}
		\end{equation}
		From \eqref{kappadef} and \eqref{ADJSO}, there exists a constant $C>0$ so that  \begin{align}\label{energy}
			\begin{aligned}
				&\,\mathbb{E}|\phi(0)|^2_{L^2(G)}+ |\phi|^2_{L_{\mathcal{F}}^2\left(0, T/2 ; L^2(G)\right)}+ |\nabla\phi|^2_{L_{\mathcal{F}}^2\left(0, T/2 ; L^2(G;\mathbb{R}^N)\right)}\\
				&\leq C\left[\frac{1}{T^2}|\phi|^2_{L_{\mathcal{F}}^2(T/2, 3T/4 ; L^2(G))} +|h|^2_{L_{\mathcal{F}}^2(0, 3T/4 ; L^2(G))} \right].
			\end{aligned}
		\end{align}
		Since  the functions $\overline{\theta}$ and $\overline{\varphi}$ (resp., $\theta$ and $\varphi$ ) are bounded in $(0, T/2)\times\overline{G}$ (resp., $(T/2, 3T/4)\times\overline{G}$), we conclude that
		\begin{align*}
			\begin{aligned}
				&\,\mathbb{E}|\phi(0)|^2_{L^2(G)}+\mathbb{E}\int_0^{T/2}\int_G\overline{\theta}^2\overline{\varphi}^3\phi^2  \,dx\,dt+\mathbb{E}\int_0^{T/2}\int_G\overline{\theta}^2\overline{\varphi} |\nabla\phi|^2  \,dx\,dt\\ 
				&\leq C\left[\mathbb{E}\int_{T/2}^{3T/4}\int_G  \theta^2\varphi^3\phi^2 dx\,dt +  \mathbb{E}\int_{0}^{3T/4}\int_G h^2 \,dx\,dt\right].
			\end{aligned}
		\end{align*}
		Then, it follows that
		\begin{align}\label{energy2}
			\begin{aligned}
				&\,\mathbb{E}|\phi(0)|^2_{L^2(G)}+\mathbb{E}\int_0^{T/2}\int_G\overline{\theta}^2\overline{\varphi}^3\phi^2  \,dx\,dt+\mathbb{E}\int_0^{T/2}\int_G\overline{\theta}^2\overline{\varphi} |\nabla\phi|^2  \,dx\,dt\\
				&\leq C \left[\lambda^3\mathbb{E}\int_Q\theta^2\varphi^3 \phi^2\,dx\,dt+\lambda\mathbb{E}\int_Q \theta^2\varphi|\nabla \phi|^2\,dx\,dt +  \mathbb{E}\int_{0}^{3T/4}\int_G h^2 \,dx\,dt \right].
			\end{aligned}
		\end{align}
		Note that $\theta = \overline{\theta}$ and $\varphi = \overline{\varphi}$ on $(T/2, T)$. Therefore, we have that \begin{align}\label{Eq5.18}
			\begin{aligned}
				&\,\mathbb{E}\int_{T/2}^T\int_G\overline{\theta}^2\overline{\varphi}^3\phi^2  \,dx\,dt+\mathbb{E}\int_{T/2}^T\int_G\overline{\theta}^2\overline{\varphi} |\nabla\phi|^2  \,dx\,dt\\
				&\leq C \left[\lambda^3\mathbb{E}\int_Q\theta^2\varphi^3 \phi^2\,dx\,dt+\lambda\mathbb{E}\int_Q \theta^2\varphi|\nabla \phi|^2\,dx\,dt\right].
			\end{aligned}
		\end{align}
		From \eqref{energy2} and \eqref{Eq5.18}, we get
		\begin{align}\label{Ineq1}
			\begin{aligned}
				&\,\mathbb{E}|\phi(0)|^2_{L^2(G)}+\mathbb{E}\int_{Q}\overline{\theta}^2\overline{\varphi}^3\phi^2  \,dx\,dt+\mathbb{E}\int_{Q}\overline{\theta}^2\overline{\varphi} |\nabla\phi|^2  \,dx\,dt\\
				&\leq C \left[\lambda^3\mathbb{E}\int_Q\theta^2\varphi^3 \phi^2\,dx\,dt+\lambda\mathbb{E}\int_Q \theta^2\varphi|\nabla \phi|^2\,dx\,dt +  \mathbb{E}\int_{0}^{3T/4}\int_G h^2 \,dx\,dt \right].
			\end{aligned}
		\end{align}
		Using the system satisfied by $h$, we can easily deduce that
		\begin{align}\label{Ineq2}
			\mathbb{E}\int_{0}^{3T/4}\int_G h^2 \,dx\,dt \leq C\left(\frac{\alpha_1^2}{\beta_1^2}+ \frac{\alpha_2^2}{\beta_2^2}\right)\,\mathbb{E}\int_{Q}\overline{\theta}^2\overline{\varphi}^3\phi^2  \,dx\,dt.
		\end{align}
		Combining \eqref{Ineq1} and \eqref{Ineq2} and choosing a large enough $\beta_1$ and  $\beta_2$,  we end up with
		\begin{align*} &\,\mathbb{E}|\phi(0)|^2_{L^2(G)}+\mathbb{E}\int_{Q}\overline{\theta}^2\overline{\varphi}^3\phi^2  \,dx\,dt+\mathbb{E}\int_{Q}\overline{\theta}^2\overline{\varphi} |\nabla\phi|^2  \,dx\,dt\\
			&\leq C \left[\lambda^3\mathbb{E}\int_Q\theta^2\varphi^3 \phi^2\,dx\,dt+\lambda\mathbb{E}\int_Q \theta^2\varphi|\nabla \phi|^2\,dx\,dt\right].
		\end{align*}
		It follows that
		\begin{align}\label{findIneq}
			\begin{aligned}
				&\,\mathbb{E}|\phi(0)|^2_{L^2(G)}+\mathbb{E}\int_{Q}\overline{\theta}^2\overline{\varphi}^3\phi^2  \,dx\,dt+\mathbb{E}\int_{Q}\overline{\theta}^2\overline{\varphi} |\nabla\phi|^2  \,dx\,dt\\
				&+\mathbb{E}\int_{Q} \theta^2 \varphi^3 h^2  \,dx\,dt+\mathbb{E}\int_{Q} \theta^2 \varphi |\nabla h|^2  \,dx\,dt\\
				&\leq C \bigg[\lambda^3\mathbb{E}\int_Q\theta^2\varphi^3 \phi^2\,dx\,dt+\lambda\mathbb{E}\int_Q \theta^2\varphi|\nabla \phi|^2\,dx\,dt\\
				&\qquad\;+\lambda^3\mathbb{E}\int_{Q} \theta^2 \varphi^3 h^2  \,dx\,dt+\lambda\mathbb{E}\int_{Q} \theta^2 \varphi |\nabla h|^2  \,dx\,dt\bigg].
			\end{aligned}
		\end{align}
		Finally, this with \eqref{Carlem5.9} provide the desired Carleman estimate, and  this completes the proof.
	\end{proof}
	\section{Null controllability result}\label{section4} 
	This section is devoted to proving the main null controllability result given in Theorem \ref{th4.1SN}. To achieve this, we shall first establish the following observability inequality.
	\begin{prop}\label{Pro4.2}
		Assuming that \eqref{Assump10} holds  and $\beta_i>0$, $i=1,2$,  are sufficiently large, then there exist a constant $C>0$ and a positive weight function $\rho=\rho(t)$ blowing up at $t=T$, such that, for any $\phi_T\in L^2_{\mathcal{F}_T}(\Omega;L^2(G))$, the solution $((\phi,\Phi), \psi^1, \psi^2)$ of \eqref{ADJSO1} satisfies 
		\begin{align}\label{observaineq}
			\begin{aligned}
				&\,\mathbb{E}|\phi(0)|^2_{L^2(G)}+ \sum_{i=1}^{2} \mathbb{E}\int_Q\rho^{-2}|\psi^i|^2\,dx\,dt\leq C \left[ \mathbb{E}\int_{Q_0}  \phi^2 \,\,dx\,dt+\mathbb{E}\int_{Q}\Phi^2\,dx\,dt\right].
			\end{aligned}
		\end{align}
	\end{prop}
	\begin{proof} 
		For fixed $\mu=\mu_3$ and $\lambda=\lambda_4$ (with $\mu_3$ and $\lambda_4$ given in Theorem \ref{lem4.5st}), we choose the weight function $\rho=\rho(t)=e^{-\lambda\overline{\alpha}^{\star}(t)}$.
		Notice that
		\begin{equation}\label{ito1}
			d(\rho^{-2} (\psi^i)^2)= -2\rho_t\rho^{-3} (\psi^i)^2dt+ \rho^{-2} \left[2\psi^i d \psi^i+ (d\psi^i)^2  \right],\qquad i=1,2.
		\end{equation}
		From \eqref{ito1}, we observe that the solution $((\phi,\Phi), \psi^1, \psi^2)$ of \eqref{ADJSO1} satisfies  for any $t \in (0,T)$
		\begin{align*}
			\begin{aligned}
				\mathbb{E}\int_0^t\int_G d(\rho^{-2} (\psi^i)^2)\, dx&\leq-2\mathbb{E}\int_0^t\int_G\rho_t\rho^{-3} (\psi^i)^2\, dx \,ds-2\mathbb{E}\int_0^t\int_G\rho^{-2} |\nabla\psi^i|^2\, dx \,ds\\
				&\quad\,+ 2\mathbb{E}\int_0^t\int_G\rho^{-2}\left[a_1(\psi^i)^2+\psi^i B_1\cdot\nabla\psi^i+\frac{1}{\beta_i}\mathbbm{1}_{\mathcal{O}_i}(x)\psi^i\phi\right]\, dx \,ds\\
				&\quad\,+ \mathbb{E}\int_0^t\int_G\rho^{-2} (a_2\psi^i+B_2\cdot\nabla\psi^i)^2\, dx \,ds, \qquad i=1,2.
			\end{aligned}
		\end{align*}
		By Young's inequality, it follows that
		\begin{align*}
			\mathbb{E}\int_0^t\int_G d(\rho^{-2} (\psi^i)^2)\, dx&\leq C\Bigg[\mathbb{E}\int_0^t\int_G \rho^{-2} (\psi^i)^2\, dx \,ds+ \mathbb{E}\int_{Q}\rho^{-2} \phi^2\, dx \,dt\Bigg], \qquad i=1,2.
		\end{align*}
		Using that $\psi^i(0)=0$ and the Gronwall inequality, we obtain
		\begin{align}\label{ineqwithrho}
			\mathbb{E}\int_Q\rho^{-2}|\psi^i|^2\,dx\,dt\leq C\,\mathbb{E}\int_{Q}\rho^{-2} \phi^2\, dx \,dt,\qquad i=1,2.
		\end{align}
		From \eqref{ineqwithrho}, \eqref{rec1} and \eqref{rec2}, we get that
		\begin{align*}
			\sum_{i=1}^{2} \mathbb{E}\int_Q\rho^{-2}|\psi^i|^2\,dx\,dt\leq C\,\mathbb{E}\int_{Q}\overline{\theta}^2\overline{\varphi}^3 \phi^2\, dx \,dt.
		\end{align*}
		By Carleman estimate \eqref{improvedCarl}, we deduce that
		\begin{align*}
			\mathbb{E}|\phi(0)|^2_{L^2(G)}+ \sum_{i=1}^{2} \mathbb{E}\int_Q\rho^{-2}|\psi^i|^2\,dx\,dt \leq C\left[ \mathbb{E}\int_{Q_0}  \phi^2 \, \,dx\,dt+\mathbb{E}\int_{Q}\Phi^2\,dx\,dt\right].
		\end{align*}
		This concludes the proof of Proposition \ref{Pro4.2}.
	\end{proof} 
	Let us now establish our null controllability result for the optimality system \eqref{eqq4.7} and also deduce the proof of Theorem \ref{th4.1SN}.
	\begin{prop}\label{Lm5.5}
		Let  $\rho=\rho(t)$ be the  weight function given in Proposition \ref{Pro4.2}. Then, for any target functions $y_{i,d} \in \mathcal{H}_{i,d}$, $i=1,2,$ satisfying \eqref{inqAss11SN} and  $y_0 \in L^2_{\mathcal{F}_0}(\Omega;L^2(G))$, there exists controls 
		$$(\widehat{u}_1, \widehat{u}_2) \in \mathcal{U},$$ 
		with minimal norm such that the corresponding solution of \eqref{eqq4.7} satisfies that
		$$\widehat{y}(T,\cdot) =0\;\;\textnormal{in}\;\;G,\quad\mathbb{P}\textnormal{-a.s.}$$
		Furthermore, the controls $(\widehat{u}_1,\widehat{u}_2)$ fulfill that
		\begin{align}\label{costofleaders}
			\begin{aligned}
				&\,|\widehat{u}_1|^2_{ L^2_\mathcal{F}(0,T;L^2(\mathcal{O}))}+|\widehat{u}_2|^2_{L^2_\mathcal{F}(0,T;L^2(G))}\\
				&\leq C\left[\mathbb{E}|y_0|^2_{L^2(G)}+ \sum_{i=1}^{2}\alpha^2_i \mathbb{E}\iint_{(0,T)\times \mathcal{O}_{i,d}} \rho^2 y^2_{i,d} \,dx \,dt\right],
			\end{aligned}
		\end{align}
		where the positive constant $C$ depending on $G$, $\mathcal{O}_i$ $\mathcal{O}_{i,d}$, $T$, $a_1$, $a_2$, $B_1$ and $B_2$.
	\end{prop} 
	\begin{proof}
		Multiplying the solution $(y, \mathcal{Z}^1, \mathcal{Z}^2)$ of \eqref{eqq4.7} by the solution $((\phi,\Phi), \psi^1, \psi^2)$ of \eqref{ADJSO1}, where $\mathcal{Z}^i = (z^i, Z^i)$, $i=1,2$, and then taking the expectation and integrating by parts, we obtain that
		\begin{align}\label{dual}
			\begin{aligned}
				&\,\mathbb{E}\left\langle y(T),\phi_T\right\rangle_{L^2(G)}-\mathbb{E}\left\langle y_0,\phi(0)\right\rangle_{L^2(G)} \\
				&=\mathbb{E}\int_{Q_0} u_1\phi \, dx \,dt +\mathbb{E}\int_{Q} u_2\Phi \, dx \,dt +\sum_{i=1}^{2}\alpha_i \mathbb{E}\iint_{ (0,T)\times \mathcal{O}_{i,d}} y_{i,d} \psi^i dx \,dt.
			\end{aligned}
		\end{align}
		Notice that the null controllability property is equivalent to finding a control $(u_1, u_2)$ for each $y_0 \in L^2_{\mathcal{F}0}(\Omega; L^2(G))$, such that for any $\phi_T \in L^2_{\mathcal{F}_T}(\Omega; L^2(G))$, the following holds
		\begin{align*}
			\begin{aligned}
				\mathbb{E}\left\langle y_0,\phi(0)\right\rangle_{L^2(G)}+\mathbb{E}\int_{Q_0} u_1\phi \, dx \,dt +\mathbb{E}\int_{Q} u_2\Phi \, dx \,dt +\sum_{i=1}^{2}\alpha_i \mathbb{E}\iint_{ (0,T)\times \mathcal{O}_{i,d}} y_{i,d} \psi^i dx \,dt=0.
			\end{aligned}
		\end{align*}
		Let $\varepsilon > 0$ and $\phi_T \in L^2_{\mathcal{F}_T}(\Omega; L^2(G))$, and consider the following functional
		\begin{align*}
			J_{\varepsilon}(\phi_T)&= \frac{1}{2}\mathbb{E}\int_{Q_0}\phi^2 \,dx\,dt +\frac{1}{2}\mathbb{E}\int_{Q}\Phi^2 \,dx\,dt +\varepsilon\mathbb{E}|\phi_T|_{L^2(G)} \\
			&\quad+\mathbb{E}\left\langle y_0,\phi(0)\right\rangle_{L^2(G)}+ \sum_{i=1}^{2}\alpha_i \mathbb{E}\iint_{(0,T)\times \mathcal{O}_{i,d}} y_{i,d} \psi^i \,dx \,dt.
		\end{align*}
		It is not difficult to see that $J_{\varepsilon}: L^2_{\mathcal{F}_T}(\Omega;L^2(G))\longrightarrow\mathbb{R}$ is continuous and strictly convex. On the other hand, we have that for any $\delta>0$
		\begin{align*}
			\mathbb{E}\left\langle y_0,\phi(0)\right\rangle_{L^2(G)}&\geq -\frac{1}{2\delta}\mathbb{E}|y_0|^2_{L^2(G)} - \frac{\delta}{2}\mathbb{E}|\phi(0)|^2_{L^2(G)}.
		\end{align*}
		By the observability inequality \eqref{observaineq}, we get that
		\begin{align}\label{inneq1}
			\begin{aligned}
				\mathbb{E}\left\langle y_0,\phi(0)\right\rangle_{L^2(G)}&\geq  - \frac{\delta}{2}C\left[ \mathbb{E}\int_{Q_0}  \phi^2 \, \,dx\,dt+\mathbb{E}\int_{Q}\Phi^2\,dx\,dt\right]\\
				&\quad\,+ \frac{\delta}{2}\,\sum_{i=1}^2\mathbb{E}\int_{Q}\rho^{-2}|\psi^i|^2  \,dx\,dt-\frac{1}{2\delta }\mathbb{E}|y_0|^2_{L^2(G)}.
			\end{aligned}
		\end{align}
		We also have that
		\begin{align}\label{inneq2}
			\begin{aligned}
				\sum_{i=1}^{2}\alpha_i \mathbb{E}\iint_{(0,T)\times \mathcal{O}_{i,d}} y_{i,d} \psi^i \,dx \,dt&\geq -\frac{1}{2\delta}\sum_{i=1}^{2}\alpha_i^2 \mathbb{E}\iint_{(0,T)\times \mathcal{O}_{i,d}} \rho^2|y_{i,d}|^2 \,dx \,dt\\
				&\quad\,-  \frac{\delta}{2}\sum_{i=1}^{2} \mathbb{E}\int_{Q} \rho^{-2} |\psi^i|^2 \,dx \,dt.
			\end{aligned}
		\end{align}
		By combining \eqref{inneq1} and \eqref{inneq2}, and taking $\delta=\frac{1}{2C}$ (where $C$ is the same constant as in \eqref{inneq1}), we obtain
		\begin{align*}
			J_{\varepsilon}(\phi_T)\geq&\,\frac{1}{4}\left[ \mathbb{E}\int_{Q_0}  \phi^2 \, \,dx\,dt+\mathbb{E}\int_{Q}\Phi^2\,dx\,dt\right] +\varepsilon\mathbb{E}|\phi_T|_{L^2(G)} \\
			&- C\left[ \mathbb{E}|y_0|^2_{L^2(G)}  
			+ \sum_{i=1}^{2}\alpha^2_i \mathbb{E}\iint_{(0,T)\times \mathcal{O}_{i,d}} \rho^2|y_{i,d}|^2 dx \,dt\right].
		\end{align*}
		Therefore, $J_{\varepsilon}$ is coercive, and thus $J_{\varepsilon}$ admits a unique minimum $\phi^\varepsilon_T$. If $\phi^\varepsilon_T\neq0$, then we conclude that
		\begin{equation*}
			\langle J_\varepsilon'(\phi^\varepsilon_T),\phi_T\rangle_{L^2_{\mathcal{F}_T}(\Omega;L^2(G))}=0\quad \,\,\textnormal{for all }\quad\phi_T\in L^2_{\mathcal{F}_T}(\Omega;L^2(G)).
		\end{equation*}
		It follows that for all $\phi_T\in L^2_{\mathcal{F}_T}(\Omega;L^2(G))$
		\begin{align}\label{Eq51NS}
			\begin{aligned}
				&\,\mathbb{E}\int_{Q_0}\phi^{\varepsilon}\,\phi \,dx\,dt +\mathbb{E}\int_{Q}\Phi^{\varepsilon}\,\Phi \,dx\,dt+\varepsilon\frac{\mathbb{E}\left\langle \phi^\varepsilon_T,\phi_T\right\rangle_{L^2(G)}}{|\phi^\varepsilon_T|_{L^2_{\mathcal{F}_T}(\Omega;L^2(G))}}\\
				&+\mathbb{E}\left\langle y_0,\phi(0)\right\rangle_{L^2(G)}+ \sum_{i=1}^{2}\alpha_i \mathbb{E}\iint_{{(0,T)\times \mathcal{O}_{i,d}}} y_{i,d} \psi^i dx \,dt=0.
			\end{aligned}
		\end{align}
		Taking controls $(u^{\varepsilon}_{1},u^{\varepsilon}_{2})=(\phi_{\varepsilon},\Phi_{\varepsilon})$ in \eqref{dual}, and combining the resulting equality with \eqref{Eq51NS}, we find that for all $\phi_T\in L^2_{\mathcal{F}_T}(\Omega;L^2(G))$
		\begin{equation*}
			\varepsilon\frac{\mathbb{E}\left\langle \phi^\varepsilon_T,\phi_T\right\rangle_{L^2(G)}}{|\phi^\varepsilon_T|_{L^2_{\mathcal{F}_T}(\Omega;L^2(G))}}+\mathbb{E}\left\langle y_\varepsilon(T),\phi_T\right\rangle_{L^2(G)}=0.
		\end{equation*}
		Hence, we deduce that 
		\begin{equation}\label{eq53NS}
			\left|y_\varepsilon(T)\right|_{L^2_{\mathcal{F}_T}(\Omega;L^2(G))}\leq \varepsilon.
		\end{equation}
		Choosing $(\phi,\Phi)=(\phi_{\varepsilon},\Phi_{\varepsilon})$ in \eqref{Eq51NS} and applying the observability inequality \eqref{observaineq} along with Young's inequality, we deduce that
		\begin{align}\label{eq54NS}
			\begin{aligned}
				&\,|u^\varepsilon_{1}|^2_{ L^2_\mathcal{F}(0,T;L^2(\mathcal{O}))}+|u^\varepsilon_{2}|^2_{L^2_\mathcal{F}(0,T;L^2(G))}\\
				&\leq C\left[\mathbb{E}|y_0|^2_{L^2(G)}+ \sum_{i=1}^{2}\alpha^2_i \mathbb{E}\iint_{(0,T)\times \mathcal{O}_{i,d}} \rho^2 y^2_{i,d} \,dx \,dt\right].
			\end{aligned}
		\end{align}
		If $\phi^\varepsilon_T=0$, we get that
		\begin{equation}\label{Ezu.1}
			\lim\limits_{t\rightarrow 0^{+}}\frac{J_{\varepsilon}(t \,\phi_T)}{t}\geq0\quad \,\,\textnormal{for all }\quad\phi_T \in L^2_{\mathcal{F}_T}(\Omega;L^2(G)).
		\end{equation}
		Using \eqref{Ezu.1} and setting $(u^{\varepsilon}_1,u^{\varepsilon}_2) = (0,0)$, we find that the inequalities \eqref{eq53NS} and \eqref{eq54NS} are satisfied.\\
		According to \eqref{eq54NS}, there exists a subsequence (also denoted by $(u^{\varepsilon}_1,u^{\varepsilon}_2)$) such that as $\varepsilon\rightarrow0$
		\begin{align}\label{weakconvr}
			\begin{aligned}
				u^{\varepsilon}_1\longrightarrow  \widehat{u}_1\quad  \text{weakly in} \,\,\; L^2((0,T)\times\Omega;L^2(\mathcal{O}));&\\
				u^{\varepsilon}_2\longrightarrow  \widehat{u}_2\quad  \text{weakly in} \,\,\; L^2((0,T)\times\Omega;L^2(G)).
			\end{aligned}
		\end{align}
		By \eqref{weakconvr}, it is easy to see that
		\begin{equation}\label{Eq56}
			y_\varepsilon(T)\longrightarrow  \widehat{y}(T)\quad  \text{weakly in} \;\,\,  L^2_{\mathcal{F}_T}(\Omega;L^2(G)),\quad \textnormal{as}\; \varepsilon\rightarrow0,
		\end{equation}
		where $\widehat{y}$ is the solution of \eqref{eqq4.7} associated to the controls $\widehat{u}_1$ and $\widehat{u}_2$. Combining \eqref{eq53NS} and \eqref{Eq56}, we finally conclude that $$\widehat{y}(T,\cdot)=0\;\;\textnormal{in}\;\;G,\quad\mathbb{P}\textnormal{-a.s.}$$
		From \eqref{eq54NS} and \eqref{weakconvr}, we easily obtain the desired estimate \eqref{costofleaders}. This completes the proof of Proposition \ref{Lm5.5} and establishes our null controllability result  as stated in Theorem \ref{th4.1SN}.
	\end{proof} 
	\section{Additional comments and some open problems}\label{section5}
	This paper is devoted to the application of Stackelberg and Nash strategies to stochastic heat equations with Dirichlet boundary conditions in a bounded domain. The problem is solved in two steps: first, we characterize the Nash equilibrium by means of a backward equation. Then, the whole problem is reduced to the question of null controllability for a forward-backward coupled stochastic system, which is solved using the Carleman estimation approach.
	
	Now, we discuss some open problems:
	\begin{itemize}
		\item As it is mentioned in \cite{oukBouElgMan}, the introduction of two leaders represents a technical constraint. This is a classical issue when dealing with the controllability of forward stochastic parabolic equations. See \cite{Preprintelgrou23, tang2009null} for more details. Therefore, it is challenging to extend the result of the present paper to a system without using the extra control $u_2$.
		
		In contrast to forward equations,  we can handle the backward ones with only one leader $u$ and two followers $v_1$ and $v_2$. More precisely, using the same ideas as presented in this paper, we can treat Stackelberg-Nash null controllability of the following system
		\begin{equation*}
			{\quad\quad\qquad\quad\begin{cases}
					\begin{array}{ll}
						dy +\Delta y \,dt = \left[a_1y+B\cdot\nabla y+a_2Y+u\mathbbm{1}_{\mathcal{O}}+v_1\mathbbm{1}_{\mathcal{O}_1}+v_2\mathbbm{1}_{\mathcal{O}_2}\right] \,dt + Y\,dW(t)&\textnormal{in}\,\,Q,\\
						y= 0&\textnormal{on}\,\,\Sigma,\\
						y(T)=y_T &\textnormal{in}\,\,G.
					\end{array}
			\end{cases}}
		\end{equation*}
		\item It would be interesting to study Stackelberg-Nash controllability for some coupled stochastic parabolic equations. For this, we consider the following coupled backward stochastic parabolic system with one leader $u$ and two followers $v_1$ and $v_2$
		\begin{equation*}
			{\quad\quad\;\;\quad\begin{cases}
					\begin{array}{ll}
						dy_1 +\Delta y_1 \,dt = \left[a_{11}\,y_1+a_{12}\,y_2+b_1\,Y_1+u\mathbbm{1}_{\mathcal{O}}+v_1\mathbbm{1}_{\mathcal{O}_1}+v_2\mathbbm{1}_{\mathcal{O}_2}\right] \,dt + Y_1\,dW(t)&\textnormal{in}\,\,Q,\\
						dy_2 +\Delta y_2 \,dt = \left[a_{21}\,y_1+a_{22}\,y_2+b_2\,Y_2\right] \,dt + Y_2\,dW(t)&\textnormal{in}\,\,Q,\\
						y_i= 0,\quad i=1,2&\textnormal{on}\,\,\Sigma,\\
						y_i(T)=y_i^T,\quad i=1,2 &\textnormal{in}\,\,G,
					\end{array}
			\end{cases}}
		\end{equation*}
		where all the coefficients $a_{ij}$ and $b_i$ are bounded. For some results concerning Stackelberg-Nash controllability of deterministic coupled parabolic PDEs, we refer, e.g., to \cite{HSP18, HSP16}. This question will be  analyzed in a forthcoming paper.
		\item Another interesting problem is to study Stackelberg-Nash controllability for the semilinear stochastic parabolic equations
		\begin{equation*}
			{\qquad\quad\;\,\begin{cases}
					\begin{array}{lll}
						dy - \Delta y \,dt \;\,=\left[F_1(y,\nabla y)+u_1\mathbbm{1}_{\mathcal{O}}+v_1\mathbbm{1}_{\mathcal{O}_1}+v_2\mathbbm{1}_{\mathcal{O}_2}\right] \,dt+\left[F_2(y,\nabla y) +u_2\right]\,dW(t)&\textnormal{in}\,\,Q,\\
						y=0 &\textnormal{on}\,\,\Sigma,\\
						y(0)=y_0 &\textnormal{in}\,\,G,
					\end{array}
			\end{cases}}
		\end{equation*}
		where $F_1$ and $F_2$ are locally Lipschitz-continuous functions.
		\item Let $\Gamma_1$ and $\Gamma_1$ 
		be two small parts of the boundary $\Gamma$. As in \cite{AFS}, it would be interesting to treat the case where one or more controls act on the boundary. For example, it is interesting to study the Stackelberg-Nash controllability of
		the  system
		\begin{equation*}
			{\qquad\quad\quad\begin{cases}
					\begin{array}{ll}
						dy - \Delta y \,dt=\left[a_1y+B_1\cdot\nabla y+u_1\mathbbm{1}_{\mathcal{O}}\right] \,dt 
						+\left[a_2y+B_2\cdot\nabla y +u_2\right]\,dW(t)&\textnormal{in}\,\,Q,\\
						y=v_1\mathbbm{1}_{\Gamma_1}+v_2\mathbbm{1}_{\Gamma_2} &\textnormal{on}\,\,\Sigma,\\
						y(0)=y_0 &\textnormal{in}\,\,G.\end{array}
			\end{cases}}
		\end{equation*}
		\item It is possible to introduce other strategies to control systems of the kind \eqref{eqq1.1}. One of them is the so-called Stackelberg-Pareto method. Following some ideas from \cite{BoMaOuNash2}, it would be quite interesting to study the controllability property of equation \eqref{eqq1.1} when we reverse the roles of the leaders and the followers.
		\item Following  \cite{BoManOukTime21}, it is of big interest  to deal with time optimal control problems for \eqref{eqq1.1}.
	\end{itemize}

\end{document}